\documentclass[12pt]{amsart}

\usepackage{ucs}

\usepackage{amssymb}
\usepackage{amsthm}
\usepackage{amsmath}
\usepackage{latexsym}
\usepackage[cp1251]{inputenc}
\usepackage{graphicx}
\usepackage{wrapfig}
\usepackage{caption}
\usepackage{subcaption}
\usepackage{indentfirst}
\usepackage[left=2.6cm,right=2.6cm,top=2.6cm,bottom=2.6cm,bindingoffset=0cm]{geometry}
\usepackage{enumerate}
\usepackage{makecell}

\usepackage{lipsum}

\DeclareMathOperator{\cay}{Cay}

\DeclareMathOperator{\rk}{rk}

\DeclareMathOperator{\Span}{Span}

\DeclareMathOperator{\rad}{rad}

\def\r{\mathrm{right}}

\makeatletter 
\def\@seccntformat#1{\csname the#1\endcsname. } 
\def\@biblabel#1{#1.}

\makeatother

\title{Divisible design graphs from Higmanian association schemes}

\author{Grigory Ryabov}

\address{School of Mathematical Sciences, Hebei Key Laboratory of Computational Mathematics and Applications, Hebei Normal University, Shijiazhuang 050024, P. R. China}

\address{Novosibirsk State Technical University, Novosibirsk, Russia}
\email{gric2ryabov@gmail.com}

\thanks{The author was supported by the grant of The Natural Science Foundation of Hebei Province (project No.~A2023205045)}

\date{}

\newtheorem{prop}{Proposition}[section]

\newtheorem{lemm}[prop]{Lemma}
\newtheorem{theo}[prop]{Theorem}

\newtheorem{corl}[prop]{Corollary}

\theoremstyle{definition}

\newtheorem*{rem1}{Remark}

\def\tm#1{\item[{\rm (#1)}]}

\begin{document}

\vspace{\baselineskip}
\vspace{\baselineskip}

\vspace{\baselineskip}

\vspace{\baselineskip}

\begin{abstract}
An imprimitive symmetric indecomposable association scheme of rank~$5$ is said to be \emph{Higmanian}. A \emph{divisible design graph} is a graph whose adjacency matrix is an incidence matrix of a symmetric divisible design. We establish conditions which guarantee that a union of some basis relations of a Higmanian association scheme is an edge set of a divisible design graph. Further, we show that several known families of divisible design graphs can be obtained as fusions of Higmanian association schemes. Finally, using our approach we construct new infinite families of divisible design graphs.
\\
\\
\textbf{Keywords}: divisible design graphs, association schemes, difference sets.
\\
\\
\textbf{MSC}: 05B05, 05C60, 05E30. 
\end{abstract}

\maketitle

\section{Introduction}

Let $V$ be a finite set of size~$v$. A regular graph\footnote{Throughout the paper, by a graph we mean a simple undirected graph without loops.} $\Gamma$ with vertex set $V$ is defined to be a \emph{divisible design graph} (\emph{DDG} for short) if $V$ can be partitioned into the classes of the same size such that any two distinct vertices from the same class have exactly~$\lambda_1$ common neighbors and any two vertices from different classes have exactly $\lambda_2$ common neighbors for some nonnegative integers~$\lambda_1$ and~$\lambda_2$. The numbers $(v,k,\lambda_1,\lambda_2,m,n)$, where $k$ is the degree of each vertex of $\Gamma$, $m$ is the number of the above classes, and $n$ is the size of each class, are called the \emph{parameters} of~$\Gamma$. The partition of $V$ from the definition of DDG is called the \emph{canonical partition} of $\Gamma$. DDGs were introduced in~\cite{HKM} as a generalization of $(v,k,\lambda)$-graphs~\cite{Rud} which are exactly DDGs with $m=1$, or $n=1$, or $\lambda_1=\lambda_2$. A DDG is said to be \emph{proper} if $m,n>1$ and $\lambda_1\neq \lambda_2$.

DDGs provide a connection between graphs and designs. More precisely, if $\Gamma$ is a DDG, then an incidence structure whose points are vertices and blocks are neighborhoods of vertices of $\Gamma$ is a \emph{symmetric divisible design} (see~\cite{BJL} for the definition). An incidence matrix of this design coincides with an adjacency matrix of $\Gamma$. In fact, DDGs are precisely those graphs whose adjacency matrices are incidence matrices of symmetric divisible designs. DDGs are also special case of Deza graphs~\cite{EFHHH}. For more information on DDGs, we refer the readers to~\cite{CH0,HKM}.

One of the main problems concerned with DDGs is the problem of constructing new DDGs. Feasible parameters of DDGs with at most 1300 vertices are listed in~\cite{Gor}. Several constructions of DDGs can be found in~\cite{BG,BGHS,CH0,CH,CS,HKM,Kabanov,Kabanov2,KS,PS,Sh}. Some of them are based on a usage of strongly regular graphs, designs, Hadamard and weighing matrices. In the present paper, we establish how it is possible to construct DDGs from association schemes possessing special properties, namely, so-called Higmanian schemes. A \emph{Higmanian scheme} is defined to be an imprimitive symmetric (association) scheme of rank~$5$ which is indecomposable, i.e. not a wreath product of schemes. Studying of such schemes was initiated by Higman in~\cite{Hig}. He was motivated by a connection of these schemes with linked systems of strongly regular designs, finite geometries, and actions of finite simple groups. The term ``Higmanian (association) scheme'' was introduced in~\cite{KMZ}. Feasible intersection matrices and character-multiplicity tables of such schemes are listed in~\cite{Hig}. Several computational results on Higmanian schemes can be found in~\cite{KMZ}. A connection between some Higmanian schemes and linked systems of divisible designs was established in~\cite[Section~$4$]{KhS}.

Let $\mathcal{X}=(V,S)$ be a Higmanian scheme with point set $V$ and set of basis relations~$S$. The diagonal of $V^2$ is denoted by $\textbf{1}_V$. Clearly, $|S|=5$ and $\textbf{1}_V\in S$. It is easy to verify that $\mathcal{X}$ has one or two nontrivial parabolics, i.e. equivalence relations on~$V$ distinct from~$\textbf{1}_V$ and $V^2$ which are unions of some basis relations from~$S$. Moreover, if $\mathcal{X}$ has exactly two such parabolics, then they are of the forms $\textbf{1}_V\cup r$ and $\textbf{1}_V\cup r \cup s$, where $r,s\in S\setminus \{\textbf{1}_V\}$. Throughout the paper, we will consider only Higmanian schemes with two nontrivial parabolics. For short, the term ``Higmanian scheme'' means in this paper a Higmanian scheme with two nontrivial parabolics (see the exact definition in Section~$2$). In case of such scheme $\mathcal{X}$, one can order $S$ as $\{s_0,\ldots,s_4\}$ with
$$s_0=\textbf{1}_V,~s_1=e_0\setminus \textbf{1}_V,~s_2=e_1\setminus e_0,~\text{and}~n_3\leq n_4,$$
where $e_0$ and $e_1$ are nontrivial parabolics of $\mathcal{X}$ and $n_i$ is a valency of~$s_i$, $i\in\{0,\ldots,4\}$. The above ordering of~$S$ is said to be \emph{standard}. Given $i,j,k\in\{0,\ldots,4\}$, the intersection number of $\mathcal{X}$ corresponding to the basis relations $s_i$, $s_j$, and $s_k$ is denoted by $c_{ij}^k$. 

The main result of the paper is the theorem below which provides necessary and sufficient conditions for some graphs arising from a Higmanian scheme to be DDGs.

\begin{theo}\label{main}
Let $\mathcal{X}=(V,S)$ be a Higmanian scheme and $\{s_0,\ldots,s_4\}$ a standard ordering of~$S$. The graph $\Gamma=(V,E)$, where $E=s_2\cup s_i$ for some $i\in \{3,4\}$, is a divisible design graph if and only if $c_{33}^3=c_{33}^4$ and one of the following equalities holds:
\begin{equation}\label{cond1}
 \frac{1}{n_3}+\frac{1}{n_4}=\frac{1}{n_1}-\frac{1}{n_1+1};
\end{equation}

\begin{equation}\label{cond2}
\frac{n_2}{n_1+1}-\frac{2n_i}{n_{7-i}}=1.
\end{equation}
The canonical partition of $\Gamma$ is the partition into the classes of $e_1$ $($$e_0$, resp.$)$ whenever Eq.~\eqref{cond1} $($Eq.~\eqref{cond2}, resp.$)$ holds and $\Gamma$ is proper if and only if exactly one of Eqs.~\eqref{cond1} and~\eqref{cond2} holds.
\end{theo}

The parameters of the DDGs from Theorem~\ref{main} expressed via $n_i$'s are given in Lemma~\ref{parameters}. Observe that Eq.~\eqref{cond1} does not depend on~$i$ and it is invariant under interchanging $n_3$ and $n_4$. So this condition leads to two graphs $\Gamma_3$ and $\Gamma_4$ with edge sets $s_2\cup s_3$ and $s_2\cup s_4$, respectively. These graphs are \emph{partial complements} to each other in the sense of~\cite[Section~4.5]{HKM}, i.e. the adjacency matrix of $\Gamma_4$ can be obtained from the adjacency matrix of $\Gamma_3$ by taking the complement of the off-diagonal blocks corresponding to the edges between distinct classes of~$e_1$.

We show that several known families of DDGs can be constructed as fusions of Higmanian schemes satisfying the conditions from Theorem~\ref{main} (see Sections~$4$ and~$5$). We also construct new infinite families of DDGs using Theorem~\ref{main} (Section~$6$).

Denote the class of DDGs arising from Higmanian schemes satisfying $c_{33}^3=c_{33}^4$ and Eq.~\eqref{cond1} (Eq.~\eqref{cond2}, resp.) by $\mathcal{K}_1$ ($\mathcal{K}_2$, resp.). Theorem~\ref{main} implies that $\mathcal{K}_1\triangle \mathcal{K}_2$ consists of proper DDGs, whereas $\mathcal{K}_1\cap \mathcal{K}_2$ consists of $(v,k,\lambda)$-graphs. Corollaries~\ref{ddg1} and~\ref{ddsddg2} yield the observation below.

\begin{rem1}
The classes $\mathcal{K}_1\setminus \mathcal{K}_2$, $\mathcal{K}_2\setminus \mathcal{K}_1$, and $\mathcal{K}_1\cap \mathcal{K}_2$ contain infinitely many graphs.
\end{rem1}

It should be mentioned that the idea to construct designs as fusions of certain association schemes was realized in several papers. For example, partial geometric designs which are fusions of association schemes of rank~$4$ were investigated in the series of papers~\cite{Han,NOS,Xu1,Xu2}. The similar approach for studying some Deza graphs was used in~\cite{BPR}.

We finish the introduction with a brief outline of the paper. A necessary background of schemes, matrix algebras, $S$-rings (Schur rings), and divisible difference sets is given in Section~$2$. Theorem~\ref{main} is proved in Section~$3$. In Sections~$4$ and~$5$, we show that some known families of DDGs are covered by Theorem~\ref{main} and belong to the classes $\mathcal{K}_1$ and $\mathcal{K}_2$, respectively. In Section~$6$, we construct new infinite families of DDGs using Theorem~\ref{main}. The Higmanian schemes used in this case are Cayley schemes constructed from divisible difference sets. Finally, a discussion on possible DDGs arising from Higmanian schemes and not covered by Theorem~\ref{main} is given in Section~$7$.

The author would like to thank Prof. S. Goryainov, Prof. V. Kabanov, and Prof. M. Muzychuk for the discussions on the subject matters. 

\section{Preliminaries}

In this section, we provide a necessary background of schemes and their adjacency algebras, $S$-rings, and divisible difference sets. In the first and second subsections, we follow in general to~\cite{CP}, whereas in the third one to~\cite{Ry}. The material of the fourth subsection is taken from~\cite{DJ} and~\cite{Pott}.

\subsection{Schemes}
Let $V$ be a finite set. Given a binary relation $s\subseteq V^2$ and $x\in V$, put $xs=\{y\in V:~(x,y)\in s\}$ and $s^*=\{(z,y):~(y,z)\in s\}$. Given a collection $S$ of binary relations on $V^2$, put $S^*=\{s^*:~s\in S\}$. Given $r,s\subseteq V^2$, a composition of $r$ and $s$ is denoted by $rs$. Given an equivalence relation $e$ on $V$, the set of all classes of $e$ is denoted by $V/e$.

Let $S$ be a partition of $V^2$. The pair $\mathcal{X}=(V,S)$ is called an (\emph{association}) \emph{scheme} on $V$ if $\textbf{1}_{V}\in S$, $S^{*}=S$, and given $r,s,t\in S$ the number
$$c_{rs}^t=|xr\cap ys^{*}|$$
does not depend on the choice of $(x,y)\in t$. The elements of $S$, the numbers $c_{rs}^t$, and the number $\rk(\mathcal{X})=|S|$ are called the \emph{basis relations}, the \emph{intersection numbers}, and the \emph{rank} of $\mathcal{X}$, respectively. The scheme $\mathcal{X}$ is said to be \emph{symmetric} if $s^*=s$ for every $s\in S$. Every symmetric scheme is \emph{commutative}, i.e. $c_{rs}^t=c_{sr}^t$ for all $r,s,t\in S$. Given $s\in S$, the \emph{intersection matrix} $C_s$ is defined to be an $(|S|\times |S|)$-matrix whose elements are indexed by the elements of $S^2$ such that
$$(C_s)_{rt}=c_{rs}^t,~r,t\in S.$$ 

A binary relation on $V^2$ is called a \emph{relation} of $\mathcal{X}$ if it is a union of some basis relations of~$\mathcal{X}$. The set of all relations of $\mathcal{X}$ is denoted by $S^\cup$. Given $r\in S^\cup$, the number 
$$n_r=\sum \limits_{s\in S,~s\subseteq r} c_{ss^*}^{\textbf{1}_V}$$
is called the \emph{valency} of $r$. It is easy to see that $n_r=|xr|$ for every $x\in V$.

An equivalence relation $e$ on $V$ is defined to be a \emph{parabolic} of $\mathcal{X}$ if $e$ is a relation of $\mathcal{X}$ or, equivalently, $e\in S^\cup$. Every scheme on $V$ has two parabolics, namely, $\textbf{1}_V$ and $V^2$. A parabolic $e$ of $\mathcal{X}$ is said to be \emph{nontrivial} if $e\neq \textbf{1}_V$ and $e\neq V^2$. A scheme is said to be \emph{primitive} if it does not have a nontrivial parabolic and \emph{imprimitive} otherwise. If $e$ is a parabolic of $\mathcal{X}$, then each class of $e$ is of size~$n_e$ and $|V/e|=|V|/n_e$. If $e_0$ and $e_1$ are parabolics of $\mathcal{X}$ such that $e_0\subseteq e_1$, then $n_{e_0}$ divides $n_{e_1}$.

For every relation $s\in S^\cup$ of $\mathcal{X}$, the largest relation $\rad(s)\subseteq V^2$ such that $\rad(s)s=s\rad(s)=s$ is a parabolic of $\mathcal{X}$. One can check that $e\subseteq \rad(s)$ for some $s\in S^\cup$ and a parabolic $e$ of $\mathcal{X}$ if and only if 
\begin{equation}\label{radical}
(U\times W) \subseteq s~\text{for all}~U,W\in V/e~\text{such that}~(U\times W)\cap s \neq \varnothing.
\end{equation}


Throughout this paper, the scheme $\mathcal{X}=(V,S)$ is said be \emph{Higmanian} if
\begin{enumerate}

\tm{1} $\rk(\mathcal{X})=5$,

\tm{2} $\mathcal{X}$ is symmetric,

\tm{3} there are $s_1,s_2\in S$ such that $e_0=s_0\cup s_1$ and $e_1=s_0\cup s_1 \cup s_2$, where $s_0=\textbf{1}_V$, are parabolics of $\mathcal{X}$,

\tm{4} $\rad(s)=\textbf{1}_V$ for every $s\in S$ outside $e_1$.

\end{enumerate}

A Higmanian scheme is imprimitive due to Condition~$(3)$. One can verify that Conditions~$(1)$ and~$(4)$ from the above definition imply that the parabolics $e_0$ and $e_1$ from Condition~$(3)$ are the only nontrivial parabolics of $\mathcal{X}$. Condition~$(4)$ is equivalent (modulo the other ones) to that $\mathcal{X}$ is indecomposable, i.e. not a wreath product of schemes (see~\cite[Section~3.4]{CP} for the definition).

Let $e$ be a parabolic of $\mathcal{X}$. Given $s\in S$, set $s_{V/e}=\{(U,W)\in (V/e)^2: s_{U,W}\neq \varnothing\}$, where $s_{U,W}=s\cap (U\times W)$. Given $U\in V/e$, set $s_{U}=s_{U,U}$. The sets $\{s_{V/e}:s\in S\}$ and $\{s_{U}:s\in S, s_{U}\neq \varnothing\}$ are denoted by $S_{V/e}$ and $S_{U}$, respectively. The pairs
$$\mathcal{X}_{V/e}=(V/e,S_{V/e})~\text{and}~\mathcal{X}_{U}=(U,S_{U}),$$ 
where $U\in V/e$, are schemes called the \emph{quotient} of $\mathcal{X}$ modulo~$e$ and the \emph{restriction} of $\mathcal{X}$ on $U$, respectively.

If we replace Conditions~$(3)$ and~$(4)$ in the definition of a Higmanian scheme by the condition that $\mathcal{X}$ has nontrivial parabolics $e_0$ and $e_1$ such that
\begin{equation}\label{equiv1}
\rk(\mathcal{X}_{U_0})=2~\text{and}~\rk(\mathcal{X}_{U_1})=3
\end{equation}
for every $U_0\in V/e_0$ and every $U_1\in V/e_1$
and
\begin{equation}\label{equiv2}
\rk(\mathcal{X}_{V/e_1})=2,~\rk(\mathcal{X}_{V/e_0})=3,
\end{equation}
then we obtain an equivalent definition. Indeed, Condition~$(3)$ is equivalent to Eq.~\eqref{equiv1}. Due to Eq.~\eqref{radical}, Condition~$(4)$ is equivalent to that each of two basis relations of $\mathcal{X}$ outside $e_1$ has a nonempty intersection with $U\times W$ and $U^\prime\times W^\prime$ for all distinct $U,W\in V/e_0$ lying in distinct classes of~$e_1$ and all distinct $U^\prime,W^\prime\in V/e_1$. The latter is equivalent to Eq.~\eqref{equiv2}. Thus, our definition of a Higmanian scheme is equivalent to the definition of a scheme from the intersection of classes I and~II from~\cite{Hig}.

\subsection{Adjacency algebras}

If $M$ is a $\{0,1\}$-matrix of size $(|V|\times |V|)$ whose elements are indexed by the elements of~$V$, then put $s(M)=\{(x,y)\in V^2:~M_{xy}=1\}$. Let $\mathcal{M}$ be a collection of $\{0,1\}$-matrices of size~$(|V|\times |V|)$ such that $\mathcal{M}$ is closed under taking a transpose matrix, $I\in \mathcal{M}$, and
$$\sum \limits_{M\in \mathcal{M}} M=J,$$
where $I=I_{|V|}$ and $J=J_{|V|}$ are the identity and all-ones matrices, respectively, of size $(|V|\times |V|)$. Suppose that the linear space 
$$\mathcal{W}=\Span_{\mathbb{C}}\{M:~M\in \mathcal{M}\}$$ 
is closed under matrix multiplication, i.e $\mathcal{W}$ is an algebra with respect to the matrix multiplication. Then due to~\cite[Theorem~2.3.7]{CP}, the pair 
$$\mathcal{X}=\mathcal{X}(\mathcal{M})=(V,S),$$
where 
$$S=\{s(M):~M\in \mathcal{M}\},$$
is a scheme on $V$ and $\mathcal{W}$ is the \emph{adjacency algebra} of $\mathcal{X}$ (see~\cite[Section~2.3.1]{CP} for the definition). Clearly, $\rk(\mathcal{X})=|\mathcal{M}|$. It is easy to see that the intersection numbers of $\mathcal{X}$ are the structure constants of $\mathcal{W}$ with respect to the basis~$\mathcal{M}$. If $M\in \mathcal{M}$, then $s(M^T)=s(M)^*$, where $M^T$ is the transpose of $M$, and the number of nonzero entries in each row and each column of~$M$ is the same and equal to the valency $n_{s(M)}$ of $s(M)\in S$.

The set of all matrices which are sums of some matrices from $\mathcal{M}$ is denoted by $\mathcal{M}^\cup$. If $M\in \mathcal{M}^\cup$, then $s(M)$ is a relation of $\mathcal{X}$. Moreover, $s(M)$ is a parabolic of $\mathcal{X}$ if and only if $M$ is a block-diagonal matrix with all-identity diagonal blocks of the same size. Given $M\in \mathcal{M}^\cup$, let $\rad(M)$ be the $\{0,1\}$-matrix with the maximal number of nonzero entries such that $\rad(M)M=M\rad(M)=lM$ for some positive integer $l$. Then $\rad(M)$ is a block-diagonal matrix from $\mathcal{M}^\cup$ with all-identity blocks of the same size~$l$ and $\rad(s(M))=s(\rad(M))$. The above discussion implies the following lemma.  

\begin{lemm}\label{algebra}
In the above notations, $\mathcal{X}=\mathcal{X}(\mathcal{M})$ is a Higmanian scheme if and only if

\begin{enumerate}

\tm{1} $|\mathcal{M}|=5$,

\tm{2} $M^T=M$ for every $M\in \mathcal{M}$,

\tm{3} there are $M_1,M_2\in \mathcal{M}$ such that $N_0=M_0+M_1$ and $N_1=M_0+M_1+M_2$, where $M_0=I$, are block-diagonal matrices with all-identity blocks of the same size,

\tm{4} $\rad(M)=I$ for every $M\in \mathcal{M}\setminus \{M_0,M_1,M_2\}$.

\end{enumerate}

\end{lemm}

\subsection{$S$-rings}

Let $G$ be a finite group and $\mathbb{Z}G$ the integer group ring. The identity element and the set of all nonidentity elements of $G$ are denoted by~$1_G$ and~$G^\#$, respectively. A product of two elements $x=\sum_{g\in G} x_g g,y=\sum_{g\in G} y_g g\in\mathbb{Z}G$ will be written as $x\cdot y$. If $X\subseteq G$, then the element $\sum \limits_{x\in X} {x}$ of the group ring $\mathbb{Z}G$ is denoted by~$\underline{X}$. Given $X\subseteq G$, put $\rad(X)=\{g\in G:\ gX=Xg=X\}$. It is easy to verify that
\begin{equation}\label{easy}
\underline{X}\cdot \underline{Y}=\underline{Y}\cdot \underline{X}=|Y|\underline{X}
\end{equation}
for all $X,Y\subseteq G$ such that $Y\subseteq \rad(X)$. The set $\{x^{-1}:x\in X\}$ and the binary relation $\{(g,xg):~x\in X,~g\in G\}\subseteq G^2$ are denoted by $X^{-1}$ and $s(X)$, respectively. Clearly, $s(1_G)=\textbf{1}_G$. If $X$ is a nonempty identity-free inverse-closed subset of $G$, then the \emph{Cayley graph} $\cay(G,X)$ over $G$ with connection set $X$ is defined to be a graph with vertex set $G$ and edge set $s(X)$.

A subring  $\mathcal{A}\subseteq \mathbb{Z}G$ is called an \emph{$S$-ring} (a \emph{Schur ring}) (see~\cite{Schur,Wi}) over $G$ if there exists a partition $\mathcal{S}=\mathcal{S}(\mathcal{A})$ of~$G$ such that:

$(1)$ $\{1_G\}\in\mathcal{S}$;

$(2)$  if $X\in\mathcal{S}$, then $X^{-1}\in\mathcal{S}$;

$(3)$ $\mathcal{A}=\Span_{\mathbb{Z}}\{\underline{X}:\ X\in\mathcal{S}\}$.

\noindent The elements of $\mathcal{S}$ and the number $\rk(\mathcal{A})=|\mathcal{S}|$ are called the \emph{basic sets} and the \emph{rank} of~$\mathcal{A}$, respectively. The $S$-ring $\mathcal{A}$ is said to be \emph{symmetric} if $X^{-1}=X$ for every $X\in \mathcal{S}$. Let $X,Y\in\mathcal{S}$. If $Z\in \mathcal{S}$, then the number of distinct representations of $z\in Z$ in the form $z=xy$ with $x\in X$ and $y\in Y$ does not depend on the choice of $z\in Z$. Denote this number by $c^Z_{XY}$. One can see that $\underline{X}\cdot\underline{Y}=\sum \limits_{Z\in \mathcal{S}}c^Z_{XY}\underline{Z}$. Therefore the numbers $c^Z_{XY}$ are the structure constants of $\mathcal{A}$ with respect to the basis $\{\underline{X}:\ X\in\mathcal{S}\}$.

A set $X\subseteq G$ is called an \emph{$\mathcal{A}$-set} if $\underline{X}\in \mathcal{A}$ or, equivalently, $X$ is a union of some basic sets of~$\mathcal{A}$. The set of all $\mathcal{A}$-sets is denoted by $\mathcal{S}^\cup$. A subgroup $A\leq G$ is called an \emph{$\mathcal{A}$-subgroup} if $A$ is an $\mathcal{A}$-set. For every $\mathcal{A}$-set $X$, the set $\rad(X)$ is an $\mathcal{A}$-subgroup. An $S$-ring is said to be \emph{primitive} if it does not have a proper nontrivial $\mathcal{A}$-subgroup and \emph{imprimitive} otherwise.




The $S$-ring $\mathcal{A}$ is said be \emph{Higmanian} if
\begin{enumerate}

\tm{1} $\rk(\mathcal{A})=5$,

\tm{2} $\mathcal{A}$ is symmetric,

\tm{3} there are $X_1,X_2\in \mathcal{S}$ such that $A_0=X_0\cup X_1$ and $A_1=X_0\cup X_1\cup X_2$, where $X_0=\{1_G\}$, are $\mathcal{A}$-subgroups,

\tm{4} $\rad(X)=\{1_G\}$ for every $X\in \mathcal{S}$ outside $A_1$.

\end{enumerate}

A Higmanian $S$-ring is imprimitive due to Condition~$(3)$. Note that Conditions~$(1)$ and~$(4)$ from the above definition imply that the $\mathcal{A}$-subgroups $A_0$ and $A_1$ from Condition~$(3)$ are the only proper nontrivial $\mathcal{A}$-subgroups of $G$.


If $\mathcal{A}$ is an $S$-ring over $G$, then 
$$\mathcal{X}=\mathcal{X}(\mathcal{A})=(G,S),$$
where $S=\{s(X):~X\in \mathcal{S}\}$, is a \emph{Cayley scheme} over $G$, i.e. a scheme whose automorphism group has a regular subgroup isomorphic to~$G$. Due to~\cite[Theorem~2.4.16]{CP}, the mapping 
$$\mathcal{A}\mapsto \mathcal{X}(\mathcal{A})$$ 
is a partial order isomorphism between the $S$-rings and the Cayley schemes over $G$. Clearly, $\rk(\mathcal{X})=\rk(\mathcal{A})$. Given $X\subseteq G$, we have $X\in \mathcal{S}^\cup$ if and only if $s(X)\in S^\cup$. It can be verified that $n_{s(X)}=|X|$ and $c_{s(X)s(Y)}^{s(Z)}=c_{XY}^Z$ for all $X,Y,Z\in \mathcal{S}$. One can see that $H$ is an $\mathcal{A}$-subgroup if and only if $e=s(H)$ is a parabolic of $\mathcal{X}$. Finally, $s(X^{-1})=s(X)^*$ and $\rad(s(X))=s(\rad(X))$ for every $X\in \mathcal{S}$. Thus, we obtain the following lemma.

\begin{lemm}\label{ringscheme}
An $S$-ring $\mathcal{A}$ is Higmanian if and only if the scheme $\mathcal{X}(\mathcal{A})$ so is. 
\end{lemm}

\subsection{Divisible difference sets}

As in the previous subsection, $G$ is a finite group. Let $N$ be a subgroup of $G$. A non-empty subset $D$ of $G$ is called a \emph{divisible difference set} (\emph{DDS} for short) relative to~$N$ if 
$$\underline{D}\cdot\underline{D}^{-1}=ke+\lambda_1 \underline{N}^\#+\lambda_2 (\underline{G}-\underline{N}),$$
where $k=|D|$ and $\lambda_1$ and $\lambda_2$ are nonnegative integers. The numbers $(m,n,k,\lambda_1,\lambda_2)$, where $m=|G:N|$ and $n=|N|$, are called the \emph{parameters} of $D$. The DDS $D$ is said to be \emph{reversible} if $D^{-1}=D$. If $\lambda_1=0$, then $D$ is a \emph{relative difference set} (\emph{RDS} for short) with \emph{forbidden subgroup} $N$ and parameters $(m,n,k,\lambda_2)$. It is easy to check that a complement to the DDS $D$ is also a DDS with parameters~$(m^\prime,n^\prime,k^\prime,\lambda_1^\prime,\lambda_2^\prime)$, where $m^\prime=m$, $n^\prime=n$, 
$$k^\prime=mn-k,~\lambda_1^\prime=mn-2k+\lambda_1,~\text{and}~\lambda_2^\prime=mn-2k+\lambda_2.$$ 
Clearly, a complement to an RDS is not an RDS in general.

We say that the DDS~$D$ satisfies the \emph{intersection condition} if the number $|D\cap Ng|$ does not depend on $g\in G$ or, equivalently, $D$ contains the same number of elements from each $N$-coset. It is easy to see that if $D$ satisfies the intersection condition, then $G\setminus D$ also satisfies. Suppose that $D$ is an RDS. Then $D$ satisfies the intersection condition if and only if $|D\cap Ng|=1$ for every $g\in G$, i.e. $D$ is a transversal for $N$ in $G$. In this case, $D$ is called \emph{semiregular} and $k=m=\lambda n$. The equality $k=m$ for the parameters of an RDS is equivalent to a semiregularity of this RDS.  

The lemma below which establishes the connection between Cayley DDGs and DDSs follows immediately from the definitions.

\begin{lemm}\label{connectddgdds}
The following statements hold.
\begin{enumerate}

\tm{1} If $D$ is an identity-free reversible DDS in a group $G$ of order~$v$ with parameters $(m,n,k,\lambda_1,\lambda_2)$ relative to a subgroup $N\leq G$, then $\cay(G,D)$ is a DDG with parameters $(v,k,\lambda_1,\lambda_2,m,n)$ and the canonical partition into the right $N$-cosets.

\tm{2} If $\cay(G,D)$ is a Cayley DDG over a group $G$ with parameters $(v,k,\lambda_1,\lambda_2,m,n)$, then the corresponding canonical partition of $G$ is the partition into the right $N$-cosets for some subgroup $N\leq G$ and $D$ is an identity-free reversible DDS with parameters $(m,n,k,\lambda_1,\lambda_2)$ relative to $N$.
\end{enumerate}

\end{lemm}


In the following two lemmas, we provide known families of DDSs which will be used further for constructing DDGs. 


\begin{lemm}\label{dds}
Let $q$ be a prime power and $r$ and $s$ positive integers such that $r\geq s\geq 1$. Then any abelian group of order $q^{2r-s}(q^s-1)/(q-1)$ having an elementary abelian subgroup of order~$q^r$ has a divisible difference set satisfying the intersection condition with parameters
$$m=q^{2r-2s}(q^s-1)/(q-1),~n=q^s,~k=q^{2r-s-1}(q^s-1)/(q-1),$$
$$\lambda_1=q^{2r-s-1}(q^{s-1}-1)/(q-1),~\lambda_2=q^{2r-s-2}(q^s-1)/(q-1).$$
\end{lemm}

\begin{proof}
Let us recall the construction of a DDS from~\cite[Theorem~2.3.6]{Pott}. Let $H$ be a vector space of dimension~$r$ over a finite field $\mathbb{F}_q$ of order~$q$, $N$ an $s$-dimensional subspace of $H$, and $G$ an arbitrary abelian group of order~$q^{2r-s}(q^s-1)/(q-1)$ containing $H$ as a subgroup. Let $\{g_0,\ldots,g_{l-1}\}$, where $l=q^{r-s}(q^s-1)/(q-1)$, be a transversal for $H$ in $G$. There are exactly $l$ $(r-1)$-dimensional subspaces $H_0,\ldots,H_{l-1}$ of $H$ not containing~$N$. Put 
$$D=\bigcup \limits_{i\in\{0,\ldots,l-1\}} H_ig_i.$$ 
Then $D$ is a DDS with the forbidden subgroup~$N$ and the required parameters by~\cite[Theorem~2.3.6]{Pott}. It remains to verify only that $D$ satisfies the intersection condition. Let $g\in G$. Since $\{g_0,\ldots,g_{l-1}\}$ is a transversal for $H$ in $G$ and $N\leq H$, there is $i\in\{0,\ldots,l-1\}$ such that $Ng\subseteq Hg_i$ and $Ng\cap Hg_j=\varnothing$ for all $j\in \{0,\ldots,l-1\}$ with $j\neq i$. Therefore 
$$D\cap Ng=(D\cap Hg_i)\cap Ng=H_ig_i\cap Ng.$$
As $H_i$ is an $(r-1)$-dimensional subspace of $H$ not containing $N$, we have $\underline{H_i}\cdot\underline{N}=q^{s-1}\underline{H}$ and hence
$$|D\cap Ng|=|H_ig_i\cap Ng|=q^{s-1}$$
for every $g\in G$. Thus, $D$ satisfies the intersection condition. 
\end{proof}

\begin{lemm}\cite[Theorem~1]{DJ}\label{rds}
Let $i$ and $j$ be positive integers such that $i/2<j\leq i$. Then the group $\mathbb{Z}_4^j\times \mathbb{Z}_2^{i-j}$ has a semiregular relative difference set with parameters~$(2^{i},2^j,2^{i},2^{i-j})$.
\end{lemm}

\section{Proof of Theorem~\ref{main}} 

Let $\mathcal{X}=(V,S)$ be a Higmanian scheme and $S=\{s_0,s_1,s_2,s_3,s_4\}$ the standard ordering of~$S$, i.e. $s_0=\textbf{1}_V$, $e_0=s_0\cup s_1$ and $e_1=s_0\cup s_1\cup s_2$ are nontrivial parabolics of $\mathcal{X}$, and $n_{s_3}\leq n_{s_4}$. Given $i,j,k\in\{0,\ldots,4\}$, the numbers $n_{s_i}$, $c_{s_is_j}^{s_k}$, and the intersection matrix $C_{s_i}$ are denoted by $n_i$, $c_{ij}^k$, and $C_i$, respectively. We start the proof with the lemma which provides a parametrization of~$C_i$'s. 

\begin{lemm}\label{intermatr}
In the above notations,

$$C_2=\begin{pmatrix}
 0 & 0 & 1 & 0 & 0\\
 0 & 0 & \alpha-1 & 0 & 0\\
 \alpha\beta & \alpha\beta & \alpha(\beta-1)& 0&0\\
 0& 0 & 0 & \beta\gamma & \beta\gamma\\
 0&  0& 0& \beta(\alpha-\gamma)&\beta(\alpha-\gamma)\\
\end{pmatrix},$$

$$C_3=\begin{pmatrix}
 0& 0& 0& 1&0\\
 0&  0& 0& \gamma-1&\gamma\\
 0&  0& 0& \beta\gamma&\beta\gamma\\
 \gamma\delta& \frac{\delta\gamma(\gamma-1)}{\alpha-1} & \frac{\delta\gamma^2}{\alpha}& 2\varepsilon\gamma-\tau-\frac{\alpha(\varepsilon\gamma-\tau)}{\gamma}&\varepsilon\gamma-\tau\\
 0&  \frac{\delta\gamma(\alpha-\gamma)}{\alpha-1}& \frac{\delta\gamma(\alpha-\gamma)}{\alpha}& \frac{(\varepsilon\gamma-\tau)(\alpha-\gamma)}{\gamma}&\tau
\end{pmatrix},$$

$$C_4=\begin{pmatrix}
 0& 0& 0& 0& 1\\
 0&  0& 0& \alpha-\gamma&\alpha-\gamma-1\\
 0&  0& 0& \beta(\alpha-\gamma)&\beta(\alpha-\gamma) \\
 0&  \frac{\delta\gamma(\alpha-\gamma)}{\alpha-1}& \frac{\delta\gamma(\alpha-\gamma)}{\alpha}& \frac{(\varepsilon\gamma-\tau)(\alpha-\gamma)}{\gamma}& \tau\\
 \delta(\alpha-\gamma)& \frac{\delta(\alpha-\gamma)(\alpha-\gamma-1)}{\alpha-1}&  \frac{\delta(\alpha-\gamma)^2}{\alpha}& \frac{\tau(\alpha-\gamma)}{\gamma}& \varepsilon(\alpha-\gamma)-\tau
\end{pmatrix},$$
where
$$\alpha=n_1+1,~\beta=\frac{n_2}{n_1+1},~\gamma=\frac{n_3(n_1+1)}{n_3+n_4},~\delta=\frac{n_3+n_4}{n_1+1},~\varepsilon=\delta-\beta-1,~\tau=c_{43}^4.$$

\end{lemm}

\begin{proof}
The intersection matrices of a symmetric imprimitive scheme of rank~$5$ having a parabolic for which a rank of the quotient scheme is equal to~$3$ are given in~\cite[p.~213-214]{Hig}. The matrices $C_2$, $C_3$, and $C_4$ are exactly the matrices $M_4$, $M_2$, and $M_3$ from~\cite[p.~213-214]{Hig}, respectively, in which: 
\begin{enumerate}

\tm{1} the rows and columns are reordered according to the permutation~$(234)$ corresponding to the standard ordering of the basis relations; 

\tm{2} the parameters $v$, $l$, $S$, $k$, $\lambda$ are replaced by $\alpha$, $\beta$, $\gamma$, $\delta$, $\varepsilon$, respectively, to avoid an ambiguity; 

\tm{3} the parameters $\lambda$, $\mu$, $\widetilde{\lambda}$, $\widetilde{\mu}$ are taken equal to~$\delta-\beta-1$, $\delta$, $\beta-1$, $0$, respectively ($k-l-1$, $k$, $l-1$, $0$, respectively, in the notation of~\cite{Hig}), because in our case the quotient scheme of rank~$3$ is imprimitive and hence the parameters $(\lambda,\mu)$ and $(\widetilde{\lambda},\widetilde{\mu})$ are the parameters of strongly regular graphs isomorphic to a complete multipartite graph with parts of the same size and disjoint union of cliques of the same size, respectively;

\tm{4} the value of $(C_3)_{33}$ ($(M_2)_{22}$ in the notation of~\cite[Section~9]{Hig}) is taken from~\cite[p.~265]{KMZ} because there is a misprint in the expression for this entry in the original paper~\cite{Hig}. 

\end{enumerate}
\end{proof}

\begin{lemm}\label{addcond}
In the above notation, the following equalities are equivalent:
$$c_{33}^3=c_{33}^4,~c_{43}^3=c_{43}^4,~c_{44}^3=c_{44}^4,$$
\begin{equation}\label{tau}
\tau=\frac{\varepsilon\gamma(\alpha-\gamma)}{\alpha}.
\end{equation}
\end{lemm}

\begin{proof}
The statement of the lemma can be verified by an explicit computation using the expressions for the intersection numbers of $\mathcal{X}$ from Lemma~\ref{intermatr}.
\end{proof}

\begin{proof}[Proof of Theorem~\ref{main}]
Let $\Gamma=(V,E)$, where $E=s_2\cup s_i$ for some $i\in \{3,4\}$. Let $j\in \{1,\ldots,4\}$. The number of common neighbors of $x,y\in V$ such that $(x,y)\in s_j$ is equal to
\begin{equation}\label{neighbors}
|xE\cap yE|=\sum \limits_{\substack{r,t\in S\\r,t\subseteq E}} c_{rt}^{s_j}=c_{22}^j+2c_{2i}^j+c_{ii}^j
\end{equation}
(here we use that $\mathcal{X}$ is symmetric and hence commutative). Denote the sum from the right-hand side of the above equality by $c_j$. Since $e_1=s_0\cup s_1\cup s_2$ and $e_0=s_0\cup s_1$ are the only nontrivial parabolics of $\mathcal{X}$, the graph $\Gamma$ is a DDG if and only if 
\begin{equation}\label{part1}
c_1=c_2,~c_3=c_4
\end{equation}
and the canonical partition of $\Gamma$ is the partition into the classes of~$e_1$
or 
\begin{equation}\label{part2}
c_2=c_3=c_4
\end{equation}
and the canonical partition of $\Gamma$ is the partition into the classes of~$e_0$. Clearly, $\Gamma$ is a $(v,k,\lambda)$-graph if and only if $c_1=c_2=c_3=c_4$, i.e. both of Eqs.~\eqref{part1} and~\eqref{part2} hold. Therefore the rest of the proof follows from the lemma below.

\begin{lemm}\label{equivconds}
Eq.~\eqref{part1} $($Eq.~\eqref{part2}, resp.$)$ is equivalent to $c_{33}^3=c_{33}^4$ and Eq.~\eqref{cond1} $($Eq.~\eqref{cond2}, resp.$)$.
\end{lemm}

\begin{proof}
We divide the proof into two cases depending on~$i$.

\hspace{5mm}

\noindent\textbf{Case~1: $i=3$.} In this case, one can compute using Lemma~\ref{intermatr} and Eq.~\eqref{neighbors} that
\begin{equation}\label{c1234}
\begin{split}
c_1=\alpha\beta+\frac{\delta\gamma(\gamma-1)}{\alpha-1},~c_2=\alpha(\beta-1)+\frac{\delta\gamma^2}{\alpha},\\
c_3=2\beta\gamma+2\varepsilon\gamma-\tau-\frac{\alpha(\varepsilon\gamma-\tau)}{\gamma},~c_4=2\beta\gamma+\varepsilon\gamma-\tau.
\end{split}
\end{equation}
Both of Eqs.~\eqref{part1} and~\eqref{part2} require $c_3=c_4$. Using the expressions from Eq.~\eqref{c1234} for $c_3$ and~$c_4$, one can verify that $c_3=c_4$ is equivalent to Eq.~\eqref{tau} and hence to $c_{33}^3=c_{33}^4$ by Lemma~\ref{addcond}. 

Due to Eq.~\eqref{c1234}, the equality $c_1=c_2$ from Eq.~\eqref{part1} is equivalent to
\begin{equation}\label{c1c2}
\alpha^2(\alpha-1)=\delta\gamma(\alpha-\gamma).
\end{equation}
Substituting the expressions for $\alpha$, $\gamma$, and $\delta$ from Lemma~\ref{intermatr} to the above equality and reducing it, we obtain Eq.~\eqref{cond1}. Therefore the equality $c_1=c_2$ is equivalent to Eq.~\eqref{cond1}. Thus, Eq.~\eqref{part1} is equivalent to $c_{33}^3=c_{33}^4$ and Eq.~\eqref{cond1}. 

In view of Eq.~\eqref{c1234}, the equality $c_2=c_4$ from Eq.~\eqref{part2} is equivalent modulo Eq.~\eqref{tau} to 
$$\alpha(\beta-1)+\frac{\delta\gamma^2}{\alpha}=2\beta\gamma+\varepsilon\gamma-\frac{\varepsilon\gamma(\alpha-\gamma)}{\alpha}.$$
Substituting the expressions for $\alpha$, $\beta$, $\gamma$, $\delta$, and $\varepsilon$ from Lemma~\ref{intermatr} to the above equality and reducing it, we obtain Eq.~\eqref{cond2} for $i=3$. Thus, Eq.~\eqref{part2} is equivalent to $c_{33}^3=c_{33}^4$ and Eq.~\eqref{cond2} for $i=3$.

\hspace{5mm}

\noindent\textbf{Case~2: $i=4$.} Lemma~\ref{intermatr} and Eq.~\eqref{neighbors} yield that
\begin{equation}\label{c1234c}
\begin{split}
c_1=\alpha\beta+\frac{\delta(\alpha-\gamma)(\alpha-\gamma-1)}{\alpha-1},~c_2=\alpha(\beta-1)+\frac{\delta(\alpha-\gamma)^2}{\alpha},\\
c_3=2\beta(\alpha-\gamma)+\frac{\tau(\alpha-\gamma)}{\gamma},~c_4=2\beta(\alpha-\gamma)+\varepsilon(\alpha-\gamma)-\tau.
\end{split}
\end{equation}
This case is similar to the previous one. The only difference is in the usage of Eq.~\eqref{c1234c} instead of Eq.~\eqref{c1234}. One can check using the above expressions for $c_3$ and $c_4$ that $c_3=c_4$ is equivalent to Eq.~\eqref{tau} and hence to $c_{33}^3=c_{33}^4$ by Lemma~\ref{addcond}. Further, the equality $c_1=c_2$ from Eq.~\eqref{part1} is equivalent to Eq.~\eqref{c1c2} and consequently to Eq.~\eqref{cond1}. Thus, Eq.~\eqref{part1} is equivalent to $c_{33}^3=c_{33}^4$ and Eq.~\eqref{cond1}. 

Due to Eq.~\eqref{c1234c}, the equality $c_2=c_4$ from Eq.~\eqref{part2} is equivalent modulo Eq.~\eqref{tau} to 
$$\alpha(\beta-1)+\frac{\delta(\alpha-\gamma)^2}{\alpha}=2\beta(\alpha-\gamma)+\varepsilon(\alpha-\gamma)-\frac{\varepsilon\gamma(\alpha-\gamma)}{\alpha}.$$
Substituting the expressions for $\alpha$, $\beta$, $\gamma$, $\delta$, and $\varepsilon$ from Lemma~\ref{intermatr} to the above equality and reducing it, we obtain Eq.~\eqref{cond2} for $i=4$. Thus, Eq.~\eqref{part2} is equivalent to $c_{33}^3=c_{33}^4$ and Eq.~\eqref{cond2} for $i=4$.
\end{proof}
\end{proof}



\begin{lemm}\label{parameters}
In the above notation, let $\Gamma_i=(V,s_2\cup s_i)$, $i\in\{3,4\}$. Suppose that $c_{33}^3=c_{44}^3$.
\begin{enumerate}

\tm{1} If $\mathcal{X}$ satisfies Eq.~\eqref{cond1}, then $\Gamma_i$, $i\in \{3,4\}$, is a divisible design graph with parameters
$$v=1+n_1+n_2+n_3+n_4,~k=n_2+n_i,$$
$$\lambda_1=n_2+\frac{n_i(n_1n_i-n_{7-i})}{n_1(n_3+n_4)},\lambda_2=\frac{2n_2n_i}{n_3+n_4}+\frac{n_i^2(n_3+n_4-n_2-n_1-1)}{(n_3+n_4)^2},$$
$$m=1+\frac{n_3+n_4}{1+n_1+n_2},~n=1+n_1+n_2.$$

\tm{2} If $\mathcal{X}$ satisfies Eq.~\eqref{cond2} for $i\in \{3,4\}$, then $\Gamma_i$ is a divisible design graph with parameters
$$v=1+n_1+n_2+n_3+n_4,~k=n_2+n_i,$$
$$\lambda_1=n_2+\frac{n_i(n_1n_i-n_{7-i})}{n_1(n_3+n_4)},\lambda_2=n_2-n_1-1+\frac{n_i^2}{n_3+n_4},$$
$$m=1+\frac{n_2+n_3+n_4}{1+n_1},~n=1+n_1.$$
\end{enumerate}
\end{lemm}

\begin{proof}
The graph $\Gamma_i$ is a DDG by Theorem~\ref{main}. The parameters $v$ and $k$ are clear. If $\mathcal{X}$ satisfies Eq.~\eqref{cond1}, then Eq.~\eqref{part1} holds by Lemma~\ref{equivconds}. So the canonical partition of $\Gamma$ is the partition into the classes of~$e_1$. Therefore the parameter $n$ is equal to~$n_{e_1}=1+n_1+n_2$ and hence $m=v/n=1+\frac{n_3+n_4}{1+n_1+n_2}$. Due to Eq.~\eqref{part1}, we have $\lambda_1=c_1=c_2$ and $\lambda_2=c_3=c_4$. Since $c_{33}^3=c_{44}^3$, Lemma~\ref{intermatr} and Lemma~\ref{addcond} imply that one can express $\tau$ via $n_i$'s. Substituting this expression and the expressions for $\alpha$, $\beta$, $\gamma$, $\delta$ via $n_i$'s from Lemma~\ref{intermatr} to Eq.~\eqref{c1234} if $i=3$ and to Eq.~\eqref{c1234c} if $i=4$ and using Eq.~\eqref{part1}, we obtain the required expressions for $\lambda_1$ and $\lambda_2$ from Statement~$(1)$ of the lemma.

Similarly, if $\mathcal{X}$ satisfies Eq.~\eqref{cond2}, then Eq.~\eqref{part2} holds by Lemma~\ref{equivconds}. In this case, the canonical partition of $\Gamma$ is the partition into the classes of~$e_0$. Therefore $n=n_{e_0}=1+n_1$ and hence $m=v/n=1+\frac{n_2+n_3+n_4}{1+n_1}$. From Eq.~\eqref{part2} it follows that $\lambda_1=c_1$ and $\lambda_2=c_2=c_3=c_4$. Again, one can express $\alpha$, $\beta$, $\gamma$, $\delta$, and $\tau$ via $n_i$'s by Lemma~\ref{intermatr} and Lemma~\ref{addcond}. A substitution of these expressions to Eq.~\eqref{c1234} if $i=3$ and to Eq.~\eqref{c1234c} if $i=4$ and a reduction of the obtained equalities using Eq.~\eqref{part1} lead to the desired formulas for $\lambda_1$ and $\lambda_2$ from Statement~$(2)$ of the lemma.
\end{proof}

\section{DDGs from $\mathcal{K}_1$}

The family of Cayley DDGs considered in this section slightly generalizes the family of Cayley DDGs from~\cite[Theorem~3.1]{KS}. Firstly, let us construct a Higmanian $S$-ring. Let $q$ be a prime power, $r\geq 2$, and $\mathbb{F}=\mathbb{F}_{q^r}$ the field of order $q^r$. Denote the additive and multiplicative groups of $\mathbb{F}$ by $\mathbb{F}^+$ and $\mathbb{F}^\times$, respectively. Let $\sigma$ be a generator of $\mathbb{F}^\times$ and $\tau=\sigma^{q-1}$. Put $A=\mathbb{F}^+$, $G=A\rtimes \langle \tau \rangle$, where $\tau$ acts on $A$ by the multiplications, and $l=(q^r-1)/(q-1)$. Clearly, $|\tau|=l$ and $|G|=q^rl=q^r(q^r-1)/(q-1)$. 

The group $A$ has exactly $l$ maximal subgroups; each of them is of order $q^{r-1}$. Due to~\cite[Lemma~2.2]{KS}, there exists an ordering $A_0,A_1,\ldots,A_{l-1}$ of the set of all maximal subgroups of $A$ such that the set
$$S=\bigcup_{i=0}^{l-1} (A\setminus A_i)\tau^i$$
is inverse-closed. Let us define a partition of $G$ into the following subsets:
$$X_0=\{1_G\},~X_1=A_0^\#,~X_2=A\setminus A_0,~X_3=\bigcup_{i=1}^{l-1} A_i\tau^i,~X_4=\bigcup_{i=1}^{l-1} (A\setminus A_i)\tau^i=S\setminus X_2.$$

\begin{lemm}\label{ks0}
The $\mathbb{Z}$-module $\mathcal{A}=\Span_{\mathbb{Z}}\{\underline{X_i}:~i\in\{0,\ldots,4\}\}$ is a Higmanian $S$-ring over~$G$.
\end{lemm}

\begin{proof}
At first, let us verify that $\mathcal{A}$ is an $S$-ring. Clearly, $X_i=X_i^{-1}$ for $i\in\{0,1,2\}$. Since $S=S^{-1}$, we conclude that $X_4=X_4^{-1}$. Together with $X_3=G\setminus (A\cup X_4)$, this implies that $X_3=X_3^{-1}$. A straightforward computation using Eq.~\eqref{easy} in the group ring $\mathbb{Z}G$ implies that
$$\underline{X_1}^2=(q^{r-1}-1)e+(q^{r-1}-2)\underline{X_1},$$
\begin{equation}\label{x2}
\underline{X_2}^2=(q^r-q^{r-1})e+(q^r-q^{r-1})\underline{X_1}+(q^r-2 q^{r-1})\underline{X_2},
\end{equation}
$$\underline{X_1}\cdot\underline{X_2}=\underline{X_2}\cdot\underline{X_1}=(q^{r-1}-1)\underline{X_2}.$$

Since $\underline{X_3}=\underline{G}^\#-\underline{X_1}-\underline{X_2}-\underline{X_4}$, it remains to verify that $\underline{X_4}\cdot\underline{X_i},\underline{X_i}\cdot\underline{X_4}\in \mathcal{A}$, where $i\in\{1,2,4\}$. Observe that 
$$\underline{A_i}\cdot \underline{A_j}=q^{r-2}\underline{A}$$
whenever $i\neq j$ because $A_i$ and $A_j$ are maximal subgroups of $A$. Using this observation, one can compute that 
\begin{equation}\label{x4}
\underline{X_1}\cdot\underline{X_4}=\underline{X_4}\cdot\underline{X_1}=(q^{r-1}-q^{r-2})\underline{X_3}+(q^{r-1}-q^{r-2}-1)\underline{X_4},
\end{equation}
and
\begin{equation}\label{x2x4}
\underline{X_2}\cdot\underline{X_4}=\underline{X_4}\cdot\underline{X_2}=q^{r-2}(q-1)^2(\underline{X_3}+\underline{X_4}).
\end{equation}
From~\cite[Theorem~3.1]{KS} it follows that 
$$\underline{S}^2=|S|e+\lambda_1\underline{A}^\#+\lambda_2(\underline{G}-\underline{A})=|S|e+\lambda_1(\underline{X_1}+\underline{X_2})+\lambda_2(\underline{X_3}+\underline{X_4})$$
for some positive integers $\lambda_1$ and $\lambda_2$. So $\underline{S}^2\in \mathcal{A}$. Further, one can compute $\underline{X_4}^2$ in the following way:
\begin{equation}\label{x4square}
\begin{split}
\underline{X_4}^2=(\underline{S}-\underline{X_2})^2=\underline{S}^2-\underline{S}\cdot \underline{X_2}-\underline{X_2}\cdot\underline{S}+\underline{X_2}^2=\underline{S}^2-\underline{X_2}^2-2\underline{X_2}\cdot \underline{X_4}=\\
=|X_4|e+(\lambda_1-q^{r-1}(q-1))\underline{X_1}+(\lambda_1-q^{r-1}(q-2))\underline{X_2}+(\lambda_2-2q^{r-2}(q-1)^2)(\underline{X_3}+\underline{X_4}),
\end{split}
\end{equation}
where the first and third equalities follow from $S=X_2\cup X_4$, whereas the fourth one from Eqs.~\eqref{x2} and~\eqref{x2x4}. Therefore $\underline{X_4}^2\in \mathcal{A}$. Thus, $\mathcal{A}$ is an $S$-ring.  

Now let us show that $\mathcal{A}$ is Higmanian. Clearly, $\rk(\mathcal{A})=5$. One can see that $\mathcal{A}$ is symmetric because $X_i=X_i^{-1}$ for every $i\in\{0,\ldots,4\}$. The groups $A_0=X_0\cup X_1$ and $A_1=X_0\cup X_1 \cup X_2$ are $\mathcal{A}$-subgroups. Since $\langle X_3 \rangle=\langle X_4 \rangle=G$, the groups $A_0$ and $A_1$ are the only proper nontrivial $\mathcal{A}$-subgroups. Assume that $|\rad(X_4)|>1$. Then $\rad(X_4)\geq A_0$ because $A_0$ is the least nontrivial $\mathcal{A}$-subgroup. We obtain a contradiction to Eqs.~\eqref{easy} and~\eqref{x4}. Therefore $|\rad(X_4)|=1$. As $X_3=G\setminus (A\cup X_4)$, we conclude that $|\rad(X_3)|=1$. Thus, all the conditions from the definition of a Higmanian $S$-ring are satisfied and consequently $\mathcal{A}$ is a Higmanian $S$-ring. \end{proof}

The scheme $\mathcal{X}=\mathcal{X}(\mathcal{A})$ is a Higmanian scheme by Lemma~\ref{ringscheme}. Put $s_i=s(X_i)$, $n_i=n_{s_i}$, and $c_{ij}^k=c_{s_is_j}^{s_k}$, $i,j,k\in\{0,\ldots,4\}$. Since $s_0=\textbf{1}_G$, $e_0=s(A_0)=s_0\cup s_1$ and $e_1=s(A_1)=s_0\cup s_1\cup s_2$ are nontrivial parabolics of $\mathcal{X}$, and
$$n_3=|X_3|=q^{r-1}(l-1)\leq q^{r-1}(q-1)(l-1)=|X_4|=n_4,$$
the ordering $S=\{s_0,\ldots,s_4\}$ of $S$ is standard.

\begin{lemm}\label{ks1}
In the above notation, the following statements hold.
\begin{enumerate}

\tm{1} $c_{33}^3=c_{33}^4$.

\tm{2} $\mathcal{X}$ satisfies Eq.~\eqref{cond1}.

\tm{3} $\mathcal{X}$ satisfies Eq.~\eqref{cond2} for $i=3$ if and only if $q=3$.

\tm{4} $\mathcal{X}$ does not satisfy Eq.~\eqref{cond2} for $i=4$.

\end{enumerate}

\end{lemm}

\begin{proof}
The equality $c_{33}^3=c_{33}^4$ is equivalent to~$c_{44}^3=c_{44}^4$ by Lemma~\ref{addcond}. Clearly, the latter one is equivalent to $c_{X_4X_4}^{X_3}=c_{X_4X_4}^{X_4}$ which immediately follows from Eq.~\eqref{x4square}. Thus, Statement~$(1)$ of the lemma holds.

One can easily compute that
$$n_1=|X_1|=q^{r-1}-1,~n_2=|X_2|=q^{r-1}(q-1),$$
$$n_3=|X_3|=q^{r-1}(l-1)=q^{r-1}\left(\frac{q^r-1}{q-1}-1\right),$$
$$n_4=|X_4|=q^{r-1}(q-1)(l-1)=q^{r-1}(q-1)\left(\frac{q^r-1}{q-1}-1\right).$$
Further,
$$\frac{1}{n_3}+\frac{1}{n_4}=\frac{1}{q^{r-1}(q^{r-1}-1)}=\frac{1}{n_1}-\frac{1}{n_1+1}$$
and hence $\mathcal{X}$ satisfies Eq.~\eqref{cond1} as required in Statement~$(2)$ of the lemma. Eq.~\eqref{cond2} is equivalent to $q-1-\frac{2}{q-1}=1$ if $i=3$ and to $q-1-2(q-1)=1$ if $i=4$. Obviously, the former equality holds if and only if $q=3$ and the latter equality does not hold. Therefore Statements~$(3)$ and~$(4)$ of the lemma hold.
\end{proof}

Let $\Gamma_i=\cay(G,X_2\cup X_i)$, $i\in\{3,4\}$. Then $s_2\cup s_i$ is an edge set of $\Gamma_i$. Theorem~\ref{main}, Lemma~\ref{parameters}, Lemma~\ref{ks0}, and Lemma~\ref{ks1} imply the next corollary.

\begin{corl}\label{ddg1}
In the above notation, the following statements hold.

\begin{enumerate}

\tm{1} The graph $\Gamma_3$ is a divisible design Cayley graph from $\mathcal{K}_1$ with parameters
$$v=q^r\left(\frac{q^r-1}{q-1}\right),~k=q^{r-1}\left(\frac{q^r-1}{q-1}+q-2\right),$$
$$\lambda_1=q^{r-1}(q-1)+\frac{q^r(q^{r-2}-1)}{q-1},~\lambda_2=q^{r-2}\left(2q+\frac{q^r-1}{q-1}-4\right),$$
$$m=\left(\frac{q^r-1}{q-1}\right),~n=q^r.$$
Moreover, $\Gamma_3\in\mathcal{K}_2$, i.e. $\Gamma_3$ is a $(v,k,\lambda)$-graph, if and only if $q=3$.

\tm{2} The graph $\Gamma_4$ is a proper divisible design Cayley graph from $\mathcal{K}_1$ with parameters
$$v=q^r\left(\frac{q^r-1}{q-1}\right),~k=q^{r-1}\left(q^r-1\right),$$
$$\lambda_1=q^{r-1}\left(q^r-q^{r-1}-1\right),~\lambda_2=q^{r-2}(q-1)(q^r-1),$$
$$m=\left(\frac{q^r-1}{q-1}\right),~n=q^r.$$

\end{enumerate}

\end{corl}

The corollary below follows from Lemma~\ref{connectddgdds} and Corollary~\ref{ddg1}.

\begin{corl}
There are reversible identity-free divisible difference sets over $G$ with parameters~$(m,n,k,\lambda_1,\lambda_2)$ from Statements~$(1)$ and~$(2)$ of Corollary~\ref{ddg1}.
\end{corl}

The graphs $\Gamma_3$ and $\Gamma_4$ are the partial complements to each other. The graph $\Gamma_4$ was described in~\cite{KS} for the first time in slightly different terms. It should be mentioned that the constructions of $\Gamma_3$ and $\Gamma_4$ depend on the ordering of the set of maximal subgroups of $A$ and different orderings may lead to nonisomorphic graphs. The question on a necessary and sufficient condition of isomorphism of two DDGs constructed from two different orderings was posed by Kabanov at the XV School-Conference on Group Theory (Ekaterinburg, 2024).

Let $q=2$ and $r\in\{2,3\}$. If $r=2$, then the graphs $\Gamma_3$ and $\Gamma_4$ have parameters
$$(12,6,2,3,3,4).$$
Due to~\cite{PS} (see also~\cite{KS}), there is a unique DDG up to isomorphism with these parameters. So $\Gamma_3$ and $\Gamma_4$ are isomorphic in this case. If $r=3$, then $\Gamma_3$ and $\Gamma_4$ have parameters
$$(56,28,12,14,7,8).$$
From~\cite{KS} it follows that there are three different orderings of maximal subgroups of $A$ which can be used for constructing $\Gamma_3$ and $\Gamma_4$. It can be verified by computer calculations using~\cite{GAP} that $\Gamma_3$ and $\Gamma_4$ are isomorphic for two of these orderings and nonisomorphic for the other one. So the constructions of $\Gamma_3$ and $\Gamma_4$ lead to four pairwise nonisomorphic Cayley DDGs which is consistent with the computational results~\cite{GS}.

Let $q=3$ and $r=2$. Then $\Gamma_3$ is a $(36,15,6)$-graph, whereas $\Gamma_4$ has parameters
$$(36,24,15,16,4,9).$$
 Note that there are $32548$ pairwise nonisomorphic $(36,15,6)$-graphs (see~\cite{BM}) and~$87$ pairwise nonisomorphic DDGs with the above parameters three of which are Cayley graphs (see~\cite{GS,PS}).

\section{DDGs from $\mathcal{K}_2$}

In this section, we prove that the DDGs from~\cite[Construction~20]{PS} belong to $\mathcal{K}_2$. At first, let us describe a construction of a Higmanian scheme. A \emph{$(n,k)$-weighing matrix} $W$ of order $n$ and weight $k$ is defined to be a $(0,1,-1)$-matrix of size~$(n\times n)$ such that 
$$WW^T=kI_n.$$ 
Let $t$ be a positive integer such that there exists a symmetric $(4t,4(t-1))$-weighing matrix $W$ the main diagonal of which consists of the blocks of zeros of size~$4$. It can be verified that there is no such a matrix for $t\in\{1,2\}$ (see also~\cite{PS}) and hence we may assume that $t\geq 3$. Let us construct a matrix $W^\prime$ by replacing in~$W$ each~$0$ by~$O_2$, each~$1$ by~$I_2$, and each~$-1$ by~$J_2-I_2$, where $O_2$, $I_2$, and $J_2$ are the all-zeros, identity, and all-ones matrices of size~$(2\times 2)$, respectively. Note that $W^\prime$ is symmetric because $W$ so is. Since the parameters of $W$ are $(4t,4(t-1))$, each element of $W$ lying outside the union of all blocks of zeros of size~$4$ from the main diagonal is nonzero. Therefore every row and every column of $W^\prime$ contain exactly $4(t-1)$ nonzero elements.

Let us define the following matrices:
$$M_0=I_{8t},~M_1=I_{4t}\otimes (J_2-I_2),~M_2=I_t\otimes ((J_4-I_4)\otimes J_2),~M_3=W^\prime,~M_4=J_{8t}-\sum_{i=0}^{3} M_i,$$
where $\otimes$ denotes the Kronecker product of matrices. Put $\mathcal{M}=\{M_i:~i\in\{0,\ldots,4\}\}$. Since $W$ and hence $W^\prime$ are symmetric, $M_i=M_i^T$ for every $i\in\{0,\ldots,4\}$. In addition, $M_0=I_{8t}\in \mathcal{M}$ and 
$$\sum \limits_{i\in\{0,\ldots,4\}} M_i=J_{8t}.$$

\begin{lemm}
The linear space $\mathcal{W}=\Span_{\mathbb{C}}\{M_i:~i\in\{0,\ldots,4\}\}$ is an algebra with respect to the matrix multiplication.
\end{lemm}

\begin{proof}
To prove the lemma, we need to verify that $M_iM_j\in\mathcal{W}$ for all $i,j\in \{0,\ldots,4\}$. A straightforward computation implies that
$$M_1^2=M_0,~M_2^2=6M_0+6M_1+4M_2,~M_1M_2=M_2M_1=M_2.$$
Since $M_4=J_{8t}-\sum_{i=0}^{3} M_i$, it suffices to check that $M_3M_i,M_iM_3\in \mathcal{W}$ for every $i\in\{1,2,3\}$. Again, a straightforward computation using the definition of $W^\prime$ shows that
\begin{equation}\label{radm}
M_1M_3=M_3M_1=M_4,~M_2M_3=M_3M_2=3M_3+3M_4,
\end{equation}
\begin{equation}\label{m3square}
M_3^2=4(t-1)M_0+2(t-1)M_2+2(t-2)(M_3+M_4).
\end{equation}
\end{proof}

According to Section~$2.2$, one can form the scheme $\mathcal{X}=\mathcal{X}(\mathcal{M})=(V,S)$, where $V=\{1,\ldots,8t\}$ and $S=\{s(M_i):~i\in\{0,\ldots,4\}\}$.  

\begin{lemm}\label{ps1}
In the above notation, the scheme $\mathcal{X}$ is Higmanian.
\end{lemm}

\begin{proof}
Clearly, $|\mathcal{M}|=5$. As it was observed before, every matrix from $\mathcal{M}$ is symmetric. It is easy to see that the matrices $N_0=M_0+M_1$ and $N_1=M_0+M_1+M_2$ are block-diagonal matrices with all-identity blocks of sizes~$2$ and~$6$, respectively, on the main diagonal. From Eq.~\eqref{radm} it follows that $\rad(M_3)=I$. As $M_4=J_{8t}-\sum_{i=0}^{3} M_i$, this also holds for $M_4$. Thus, $\mathcal{X}$ is Higmanian by Lemma~\ref{algebra}.
\end{proof}

Put $s_i=s(M_i)$ and $n_i=n_{s_i}$, $i\in\{0,\ldots,4\}$. This ordering of $S$ is standard. Indeed, the relations $e_0=s(N_0)=s_0\cup s_1$ and $e_1=s(N_1)=s_0\cup s_1\cup s_2$ are nontrivial parabolics of $\mathcal{X}$ and $n_3=n_4=4(t-1)$ because the matrices $M_3=W^\prime$ and $M_4=J_{8t}-\sum_{i=0}^{3} M_i$ have exactly $4(t-1)$ nonzero elements in each row and in each column.

\begin{lemm}\label{ps2}
In the above notation, the following statements hold.
\begin{enumerate}

\tm{1} $c_{33}^3=c_{33}^4$.

\tm{2} $\mathcal{X}$ satisfies Eq.~\eqref{cond2} for $i=3$ and $i=4$.

\tm{3} $\mathcal{X}$ does not satisfy Eq.~\eqref{cond1}.

\end{enumerate}
\end{lemm}

\begin{proof}
Statement~$(1)$ immediately follows from Eq.~\eqref{m3square}. Recall that $n_i$ is equal to the number of nonzero elements in each row of $M_i$, $i\in\{0,\ldots,4\}$. So one can see that
$$n_1=1,~n_2=6,~n_3=n_4=4(t-1).$$ 
It is easy to verify that 
$$\frac{n_2}{n_1+1}-\frac{2n_3}{n_4}=\frac{n_2}{n_1+1}-\frac{2n_4}{n_3}=3-2=1,$$
$$\frac{1}{n_3}+\frac{1}{n_4}=\frac{1}{2(t-1)},~\text{and}~\frac{1}{n_1}-\frac{1}{n_1+1}=\frac{1}{2}.$$
Together with $t\geq 3$, this implies Statements~$(2)$ and~$(3)$ of the lemma.
\end{proof}

Let $\Gamma_i=(V,s_2\cup s_i)$, $i\in \{3,4\}$. Due to Theorem~\ref{main}, Lemma~\ref{parameters}, Lemma~\ref{ps1}, and Lemma~\ref{ps2}, we obtain the corollary below.

\begin{corl}\label{ddg2}
In the above notation, $\Gamma_3$ and $\Gamma_4$ are proper divisible design graphs from $\mathcal{K}_2$ with parameters~$(8t,4t+2,6,2t+2,4t,2)$.
\end{corl}

We finish this section with several remarks. In fact,~\cite[Construction~20]{PS} is a special case of the construction given in~\cite[Theorem~4.4]{CH}. However, there are DDGs arising from the latter construction which can not be obtained as fusions of a Higmanian scheme (see, e.g.,~\cite[Construction~21]{PS}). 

It is known that there is (at least) one $(4t,4(t-1))$-weighing matrix~$W_t$ the main diagonal of which consists of the blocks of zeros of size~$4$ for each $t\in\{3,4,5,6\}$ (see, e.g.,~\cite{Wallis}). The matrices $W_3$, $W_4$, $W_5$, and $W_6$ lead to DDGs with parameters
$$(24,14,6,8,12,2),~(32,18,6,10,16,2),~(40,22,6,12,20,2),~\text{and}~(48,26,6,14,24,2),$$
respectively. DDGs with the first and second parameters can be found in~\cite[Table~1]{PS}, whereas DDGs with the third and fourth parameters appear among Cayley-Deza graphs from~\cite{GS}. For the first and second collections of parameters, there is a unique DDG up to isomorphism (see~\cite[Table~1]{PS}). So $\Gamma_3$ and $\Gamma_4$ are isomorphic in this case. It can be verified by a computer calculation using~\cite{GAP} that the latter also holds in case of parameters~$(40,22,6,12,20,2)$. However, $\Gamma_3$ and $\Gamma_4$ are nonisomorphic in case of parameters~$(48,26,6,14,24,2)$ due to computer calculations using~\cite{GAP}. In view of~\cite{GS}, there are exactly seven pairwise nonisomorphic Cayley DDGs with these parameters. 

Very recent results~\cite[Proposition~4.4]{GHKL} imply that there are infinitely many $(4t,4(t-1))$-weighing matrix the main diagonal of which consists of the blocks of zeros of size~$4$.

\section{New family of DDGs}

In this section, we provide a new infinite family of DDGs which are Cayley graphs over generalized dihedral groups. The construction is based on DDSs in abelian groups. Let $A$ be an abelian group and $G=A\rtimes \langle b\rangle$, where $|b|=2$ and $a^b=a^{-1}$ for every $a\in A$. Clearly, $G$ is a generalized dihedral group associated with $A$. Let $D$ be a DDS in~$A$ with parameters~$(m,n,k,\lambda_1,\lambda_2)$, where $m,n>1$, relative to a proper nontrivial subgroup~$N$ of $A$. Suppose that $D$ satisfies the intersection condition. Together with $m,n>1$, this implies that $k,mn-k>1$. We may assume further that $k\leq mn-k$. Indeed, if $k>mn-k$, then we replace $D$ by $A\setminus D$ which is also a DDS in $A$ satisfying the intersection condition.  

Let us define a partition of $G$ into the following subsets:
$$X_0=\{e\},~X_1=N^\#,~X_2=A\setminus N,~X_3=Db,~X_4=(A\setminus D)b.$$
Since $m,n>1$ and $D$ satisfies the intersection condition, each of the above subsets is nonempty. 

\begin{lemm}\label{mp0}
The $\mathbb{Z}$-module $\mathcal{A}=\Span_{\mathbb{Z}}\{\underline{X_i}:~i\in\{0,\ldots,4\}\}$ is a Higmanian $S$-ring over~$G$.
\end{lemm}

\begin{proof}
In~\cite[Theorem~12.2]{MP}, it was verified that $\mathcal{A}$ is an $S$-ring in case of a dihedral group. However, the proof is the same in case of a generalized dihedral group. To make the text self-contained, we provide the proof here. Observe that $X_i=X_i^{-1}$ for every $i\in\{0,\ldots,4\}$. Indeed, this is obvious for $i\in\{0,1,2\}$, whereas this is true for $i\in\{3,4\}$ because all the elements from $G\setminus A$ are of order~$2$.  A straightforward computation using Eq.~\eqref{easy} in the group ring $\mathbb{Z}G$ implies that
$$\underline{X_1}^2=(n-1)e+(n-2)\underline{X_1},$$
$$\underline{X_2}^2=n(m-1)e+n(m-1)\underline{X_1}+n(m-2)\underline{X_2},$$
$$\underline{X_1}\cdot\underline{X_2}=\underline{X_2}\cdot\underline{X_1}=(n-1)\underline{X_2}.$$
Since $\underline{X_4}=\underline{G}^\#-\underline{X_1}-\underline{X_2}-\underline{X_3}$, it remains to verify that $\underline{X_3}\cdot\underline{X_i},\underline{X_i}\cdot\underline{X_3}\in \mathcal{A}$, where $i\in\{1,2,3\}$. As $D$ satisfies the intersection condition, there is $\mu\geq 1$ such that $|D\cap Na|=\mu$ for every $a\in A$. Using this observation, one can compute that
\begin{equation}\label{x3new}
\underline{X_1}\cdot\underline{X_3}=\underline{X_3}\cdot\underline{X_1}=(\mu-1)\underline{X_3}+\mu\underline{X_4},
\end{equation}
and
$$\underline{X_2}\cdot\underline{X_3}=\underline{X_3}\cdot\underline{X_2}=(k-\mu)(\underline{X_3}+\underline{X_4}).$$
Due to the definition of a DDS, we have
\begin{equation}\label{difsetsquare}
\underline{X_3}^2=\underline{D}\cdot \underline{D}^{-1}=ke+\lambda_1\underline{X_1}+\lambda_2\underline{X_2}.
\end{equation}
Thus, $\mathcal{A}$ is an $S$-ring.

Let us verify that $\mathcal{A}$ is Higmanian. Clearly, $\rk(\mathcal{A})=5$. As $X_i=X_i^{-1}$ for every $i\in\{0,\ldots,4\}$, we conclude that $\mathcal{A}$ is symmetric. The groups $A_0=X_0\cup X_1$ and $A_1=X_0\cup X_1\cup X_2$ are the only proper nontrivial $\mathcal{A}$-subgroups. Assume that $|\rad(X_3)|>1$. Since $X_3\subseteq Ab$, we have $\rad(X_3)\leq A$. So $\rad(X_3)\geq A_0$ because $A_0$ is the least nontrivial $\mathcal{A}$-subgroup of~$A$. However, $\rad(X_3)\geq A_0$ contradicts to Eqs.~\eqref{easy} and~\eqref{x3new}. Therefore $|\rad(X_3)|=1$. As $X_4=G\setminus (A\cup X_3)$, the latter also holds for $X_4$. Thus, all the conditions from the definition of a Higmanian $S$-ring are satisfied and hence $\mathcal{A}$ is a Higmanian $S$-ring.  
\end{proof}

The scheme $\mathcal{X}=\mathcal{X}(\mathcal{A})$ is a Higmanian scheme by Lemma~\ref{ringscheme}. Put $s_i=s(X_i)$, $n_i=n_{s_i}$, and $c_{ij}^k=c_{s_is_j}^{s_k}$, $i,j,k\in\{0,\ldots,4\}$. Clearly, $s_0=\textbf{1}_G$ and $e_0=s(A_0)=s_0\cup s_1$ and $e_1=s(A_1)=s_0\cup s_1\cup s_2$ are nontrivial parabolics of $\mathcal{X}$. Since $k\leq mn-k$, we have $|X_3|\leq |X_4|$ and hence $n_3\leq n_4$. So the ordering $S=\{s_0,\ldots,s_4\}$ of $S$ is standard. Put also $k_3=k$ and $k_4=mn-k$.

\begin{lemm}\label{mp1}
In the above notation, the following statements hold.
\begin{enumerate}
\tm{1} $c_{33}^3=c_{33}^4$.

\tm{2} $\mathcal{X}$ satisfies Eq.~\eqref{cond1} if and only if 
\begin{equation}\label{difset1}
\frac{1}{k}+\frac{1}{mn-k}=\frac{1}{n-1}-\frac{1}{n}.
\end{equation}

\tm{3} $\mathcal{X}$ satisfies Eq.~\eqref{cond2} for $i\in\{3,4\}$ if and only if 
\begin{equation}\label{difset2}
k_{7-i}=2n.
\end{equation}

\tm{4} $\mathcal{X}$ satisfies Eqs.~\eqref{cond1} and~\eqref{cond2} for some $i\in\{3,4\}$ if and only if $D$ is a relative difference set with parameters~$(4,2,4,2)$.

\end{enumerate}

\end{lemm}

\begin{proof}
From Eq.~\eqref{difsetsquare} it follows that $c_{33}^3=c_{33}^4=0$ and hence Statement~$(1)$ of the lemma holds. One can see that that
$$n_1=|X_1|=|N|-1=n-1,~n_2=|X_2|=|A|-|N|=mn-n=n(m-1),$$
$$n_3=|X_3|=|D|=k,~n_4=|X_4|=|A|-|D|=mn-k.$$
Now Statements~$(2)$ and~$(3)$ can be verified by a straightforward substitution of the above expressions for $n_i$'s to Eqs.~\eqref{cond1} and~\eqref{cond2}.

Let us prove Statement~$(4)$ of the lemma. In view of Statements~$(2)$ and~$(3)$, we may replace Eqs.~\eqref{cond1} and~\eqref{cond2} in Statement~$(4)$ by Eqs.~\eqref{difset1} and~\eqref{difset2}, respectively. Clearly, $(m,n,k)=(4,2,4)$ satisfy Eqs.~\eqref{difset1} and~\eqref{difset2} for $i=3$ and $i=4$. Conversely, suppose that Eqs.~\eqref{difset1} and~\eqref{difset2} for some $i\in\{3,4\}$ hold for $m,n,k$. Then substituting one of the expressions $k=2n$ or $mn-k=2n$ to Eq.~\eqref{difset1} and reducing the equality, we obtain
\begin{equation}\label{technical}
\frac{3-n}{2n(n-1)}=\frac{1}{mn-k}~\text{or}~\frac{3-n}{2n(n-1)}=\frac{1}{k},
\end{equation} 
respectively. This implies that $3-n>0$. Due to~$n>1$, we conclude that $n=2$. If $k=2n$, then $k=4$ and the first part of Eq.~\eqref{technical} yields that $mn-k=2m-k=4$. So $m=4$. If $mn-k=2n$, then $mn-k=4$ and the second part of Eq.~\eqref{technical} yields that $k=4$. Again, $m=(mn-k+k)/n=(4+k)/n=4$. Therefore $(m,n,k)=(4,2,4)$ in both cases. Since $D$ satisfies the intersection condition and $|D|=k=4=m=|A:N|$, the set $D$ is a transversal for $N$ in $A$. Thus, $D$ is a semiregular RDS and hence $\lambda_1=0$ and $\lambda_2=m/n=2$.
\end{proof}

Let $\Gamma_i=\cay(G,X_2\cup X_i)$, $i\in\{3,4\}$. Then $s_2\cup s_i$ is an edge set of $\Gamma_i$. Theorem~\ref{main}, Lemma~\ref{parameters}, Lemma~\ref{mp0}, and Lemma~\ref{mp1} imply the next corollary. To avoid an ambiguity appearing due to the same standard notations of parameters of a DDS and a DDG, we denote the parameters of a DDG by $(v^\prime,k^\prime,\lambda_1^\prime,\lambda_2^\prime,m^\prime,n^\prime)$ in the corollary below.

\begin{corl}\label{ddg3}
In the above notation, the following statements hold.

\begin{enumerate}

\tm{1} If Eq.~\eqref{difset1} holds and $(m,n,k)\neq (4,2,4)$, then $\Gamma_i$, $i\in\{3,4\}$, is a proper divisible design Cayley graph from $\mathcal{K}_1$ with parameters
$$v^\prime=2mn,~k^\prime=n(m-1)+k_i,$$
$$\lambda_1^\prime=n(m-1)+\frac{k_i(k_i-m)}{(n-1)m},\lambda_2^\prime=\frac{2(m-1)k_i}{m},$$
$$m^\prime=2,~n^\prime=mn.$$

\tm{2}  If Eq.~\eqref{difset2} holds for $i\in\{3,4\}$ and $(m,n,k)\neq (4,2,4)$, then $\Gamma_i$ is a proper divisible design Cayley graph from $\mathcal{K}_2$ with parameters
$$v^\prime=2mn,~k^\prime=n(m-1)+k_i,$$
$$\lambda_1^\prime=n(m-1)+\frac{k_i(k_i-m)}{(n-1)m},\lambda_2^\prime=n(m-2)+\frac{k_i^2}{mn},$$
$$m^\prime=2m,~n^\prime=n.$$

\tm{3} If $(m,n,k)=(4,2,4)$, then $\Gamma_3$ and $\Gamma_4$ are $(16,10,6)$-graphs. 

\end{enumerate}
 
\end{corl}

It should be mentioned that there is a unique up to isomorphism $(16,10,6)$-graph which is the Clebsch graph (see~\cite[Section~10.7]{BM}). So $\Gamma_3 $ and $\Gamma_4$ are isomorphic to this graph whenever $(m,n,k)=(4,2,4)$.

Further, we are going to use Corollary~\ref{ddg3} and DDSs from Section~$2.4$ for constructing DDGs. 

\begin{corl}\label{ddsddg1}
Let $q$ be a prime power, $r\geq 2$, and $A$ an abelian group of order $q^{r+1}(q^{r-1}-1)/(q-1)$ having an elementary abelian subgroup of order~$q^r$. Then there are divisible design Cayley graphs from $\mathcal{K}_1$ over the generalized dihedral group associated with $A$ with parameters
$$v=2q^{r+1}\left(\frac{q^{r-1}-1}{q-1}\right),~k=q^{r-1}\left(\frac{q(q+1)(q^{r-1}-1)}{q-1}-1\right),$$
$$\lambda_1=q^{r-1}\left(q^r+q^{r-1}+q-1+\frac{2(q^{r-1}-q^2)}{q-1}\right),~\lambda_2=2q^{r-2}\left(\frac{q^{r+1}-q^2}{q-1}-1\right),$$
$$m=2,~n=q^{r+1}\left(\frac{q^{r-1}-1}{q-1}\right),$$
and 
$$v=2q^{r+1}\left(\frac{q^{r-1}-1}{q-1}\right),~k=q^{r-1}\left(\frac{q(2q-1)(q^{r-1}-1)}{q-1}-1\right),$$
$$\lambda_1=q^{r-1}\left(2q^r-q-1+\frac{q^{r-1}-q^2}{q-1}\right),~\lambda_2=2q^{r-2}\left(q^{r+1}-q^2-q+1\right),$$
$$m=2,~n=q^{r+1}\left(\frac{q^{r-1}-1}{q-1}\right).$$
These divisible design graphs are proper if and only if $q\neq 2$ or $r\neq 2$.
\end{corl}

\begin{proof}
Due to Lemma~\ref{dds} applied to $r$ and $s=r-1$, the group $A$ has a DDS~$D$ satisfying the intersection condition with the first three parameters
$$m=q^2(q^{r-1}-1)/(q-1),~n=q^{r-1},~k=q^r(q^{r-1}-1)/(q-1).$$
It can be verified straightforwardly that the above~$m$, $n$, and~$k$ satisfy Eq.~\eqref{difset1} and $(m,n,k)=(4,2,4)$ if and only if $q=2$ and $r=2$. Thus, we are done by Statements~$(1)$ and~$(3)$ of Corollary~\ref{ddg3}. 
\end{proof}

The graphs from Corollary~\ref{ddsddg1} are the partial complements to each other. Suppose that $r=2$ in Corollary~\ref{ddsddg1}. If $q=2$, then both of the graphs from this corollary are isomorphic to the Clebsch graph. If $q=3$, then the first and second graphs from Corollary~\ref{ddsddg1} have parameters 
$$(54,33,24,16,2,27)~\text{and}~(54,42,33,32,2,27),$$ 
respectively, and can be found in~\cite{GS} among small Cayley-Deza graphs.

\begin{corl}\label{ddsddg2}
Let $r\geq 3$ and $A\cong \mathbb{Z}_4^{r-1}\times \mathbb{Z}_2$. Then there is a proper divisible design Cayley graph from $\mathcal{K}_2$ over the generalized dihedral group associated with $A$ with parameters
$$v=2^{2r},~k=2^{2r}-2^r-2^{r-1},$$
$$\lambda_1=2^{2r}-2^{r+1}-2^{r-1},~\lambda_2=2(2^r-1)(2^{r-1}-1),$$
$$m=2^{r+1},~n=2^{r-1}.$$
\end{corl}

\begin{proof}
From Lemma~\ref{rds} applied to $i=r$ and $j=r-1$ it follows that the group $A$ has a semiregular RDS~$D$ with parameters
$$(2^{r},2^{r-1},2^{r},2).$$
It is easy to see that $m=k=2^r=2n$ and hence $m$, $n$, $k$ satisfy Eq.~\eqref{difset2} for $i=4$. As $r\geq 3$, we have $(m,n,k)\neq (4,2,4)$. Thus, we are done by Statement~$(2)$ of Corollary~\ref{ddg3}.
\end{proof}

As we can check using the repository of papers on DDGs~\cite{Pan}, the parameters of DDGs from Corollaries~\ref{ddsddg1} and~\ref{ddsddg2} were previously unknown and hence the families of the graphs so are. 

Lemma~\ref{connectddgdds}, Corollary~\ref{ddsddg1}, and Corollary~\ref{ddsddg2} imply the following statement.

\begin{corl}
There are reversible identity-free divisible difference sets over generalized dihedral groups with parameters~$(m,n,k,\lambda_1,\lambda_2)$ from Corollaries~\ref{ddsddg1} and~\ref{ddsddg2}.
\end{corl}

\section{Further discussion}

In the current paper, we deal only with the case when an edge~$E$ set of the considered graph~$\Gamma$ is $s_2\cup s_i$, $i\in\{3,4\}$. Of course, it is possible to consider other possibilities for~$E$, namely, when $E$ is equal to one of the following sets:
$$s_1,~s_2,~s_1\cup s_2,~s_3\cup s_4,~s_1\cup s_3\cup s_4,~s_2\cup s_3 \cup s_4,$$
$$s_i,~s_1\cup s_i,~s_1\cup s_2\cup s_i,~i\in\{3,4\}.$$
A set from the first of the above rows is an edge set of a DDG $\Gamma$ if and only if $\Gamma$ is isomorphic to one of the following well-known DDGs: a disjoint union of cliques of the same size, a complete multipartite graph with parts of the same size, a disjoint union of complete bipartite graphs $K_{n,n}$ for some $n\geq 1$, and a lexicographic product of a clique and a disjoint union of edges. 

The possibilities for $E$ from the second row can be considered as in the proof of Theorem~\ref{main} using Lemma~\ref{intermatr}. One can obtain in such way that:
\begin{enumerate}

\tm{1} $s_i$, $i\in\{3,4\}$, can not be an edge set of a DDG;

\tm{2} $s_1\cup s_i$, $i\in\{3,4\}$, is an edge set of a DDG if and only if $|c_{33}^3-c_{33}^4|=2$ and one of the equalities 
$$\frac{1}{n_3}+\frac{1}{n_4}=\frac{1}{n_1-1}-\frac{1}{n_1},~\frac{n_2+1}{n_1}-\frac{2n_i}{n_{7-i}}=1$$
holds;

\tm{3} $s_1\cup s_2\cup s_i$ , $i\in\{3,4\}$, is an edge set of a DDG if and only if $|c_{33}^3-c_{33}^4|=2$,
$$n_2=n_1+1,~\text{and}~n_4=n_3n_1.$$

\end{enumerate}

We do not know whether there exists a Higmanian scheme satisfying the conditions from Statement~$(2)$ or~$(3)$ and consequently it is unclear whether these conditions can lead to constructions of DDGs. According to the computational results~\cite{HM}, there is no such scheme with at most~$34$ points. Due to~\cite[p.~214]{Hig}, the condition  $|c_{33}^3-c_{33}^4|=2$ implies that each entry of the character multiplicity table of the Higmanian scheme must be integer.

It also seems interesting to study the question when a union of some basis relations of a Higmanian scheme with exactly one nontrivial parabolic is an edge set of a DDG. We do not know even an example of such DDG. Using the computational results from~\cite{PS}, one can verify that there is no such DDG with at most~$39$ vertices.

\end{document}